\pdfoutput=1
\newif\ifspringer
\springerfalse
\documentclass[a4paper,reqno,12pt]{amsart}

\usepackage[utf8]{inputenc}
\usepackage[T1]{fontenc}
\usepackage{microtype}

\usepackage{amssymb, amsthm, amssymb, graphicx, enumerate, enumitem, color}
\usepackage[usenames,dvipsnames,svgnames,table]{xcolor}
\usepackage[foot]{amsaddr}
\usepackage{wrapfig}

\colorlet{mylinkcolor}{violet}
\colorlet{mycitecolor}{YellowOrange}
\colorlet{myurlcolor}{Aquamarine}
\usepackage[unicode=true]{hyperref}
\hypersetup{
	pdftitle={On an extremal problem for poset dimension},
	pdfauthor={Grzegorz Guśpiel, Piotr Micek, and Adam Polak},
	colorlinks,
	linkcolor={red!50!black},
	citecolor={blue!50!black},
	urlcolor={blue!80!black}
}

\newtheorem{theorem}{Theorem}

\theoremstyle{remark}

\usepackage{todonotes}

\theoremstyle{plain}
\theoremstyle{definition}

\newcommand{\set}[1]{\left\{#1\right\}}

\newcommand{\floor}[1]{{\left\lfloor #1 \right\rfloor}}
\newcommand{\ceil}[1]{{\left\lceil #1 \right\rceil}}

\newcommand{\N}[0]{\mathbb{N}}

\newcommand{\pdim}[1]{\mathrm{dim}(#1)}

\newcommand{\setR}{\mathbb{R}}

\renewcommand{\O}{\mathcal{O}}

\let\leq\leqslant
\let\geq\geqslant

\let\subset\subseteq

\let\epsilon\varepsilon

\renewenvironment{enumerate}{\begin{enumorig}[label=\textup{(\roman*)}, noitemsep, topsep=2pt plus 2pt, labelindent=.2em, leftmargin=*, widest=iii]}{\end{enumorig}}

\renewenvironment{itemize}{\begin{itemorig}[label=\textbullet, noitemsep, topsep=2pt plus 2pt, labelindent=.5em, labelsep=.5em, leftmargin=*]}{\end{itemorig}}

\makeatletter
\let\old@setaddresses\@setaddresses
\def\@setaddresses{\bigskip\bgroup\parindent 0pt\let\scshape\relax\old@setaddresses\egroup}
\makeatother

\makeatletter
\def\paragraph{\@startsection{paragraph}{4}%
  \z@\z@{-\fontdimen2\font}%
  {\normalfont\bfseries}}
\makeatother

\begin{document}

\title[]{On an extremal problem for poset dimension}

\author[G.~Guśpiel]{Grzegorz Guśpiel}
\email{guspiel@tcs.uj.edu.pl}
\author[P.~Micek]{Piotr Micek}
\author[A.~Polak]{Adam Polak}
\email{polak@tcs.uj.edu.pl}
\address[G.~Guśpiel, A.~Polak, P.~Micek]{Theoretical Computer Science Department\\
	Faculty of Mathematics and Computer Science, Jagiellonian University, Krak\'ow, Poland}
\email{piotr.micek@tcs.uj.edu.pl}

\thanks{
  Grzegorz Guśpiel was partially supported by the MNiSW grant DI2013 000443.
  Piotr Micek was partially supported by the National Science Center of Poland under grant no.\ 2015/18/E/ST6/00299.
  Adam Polak was partially supported by the Polish Ministry of Science and Higher Education program ``Diamentowy Grant''.
}

\date{\today}

\begin{abstract}
Let $f(n)$ be the largest integer such that every poset on $n$ elements has a $2$-dimensional subposet on $f(n)$ elements.
What is the asymptotics of $f(n)$?
It is easy to see that $f(n)\geq n^{1/2}$.
We improve the best known upper bound and show $f(n)=\O(n^{2/3})$.
For higher dimensions, we show $f_d(n)=\O\left(n^\frac{d}{d+1}\right)$, where
$f_d(n)$ is the largest integer such that every poset on $n$ elements has
a $d$-dimensional subposet on $f_d(n)$ elements.
\end{abstract}

\maketitle

\section{Introduction}

Every partially ordered set on $n$ elements has a chain or an antichain of size at least $n^{1/2}$,
this is an immediate consequence of Dilworth's Theorem or its easier dual counterpart.
Chains and antichains are very special instances of $2$-dimensional posets.
Surprisingly, the following simple problem is open:

{\it
Let $f(n)$ be the largest integer such that every poset on $n$ elements has a $2$-dimensional subposet on $f(n)$ elements.
What is the asymptotics of $f(n)$?
}

Although this sounds like a natural extremal-type question for posets,
it was posed only in 2010, by François Dorais~\cite{Dorais2010}.
Clearly, $n^{1/2} \leq f(n) \leq n$.
Reiniger and Yeager~\cite{Reiniger2016} proved a sublinear upper bound, that is
$f(n)=\O(n^{0.8295})$.
Their construction is a lexicographic power of standard examples.

The main idea behind our contribution was a belief that a $(k \times k)$-grid
is asymptotically the largest $2$-dimensional subposet of the
$(k \times k \times k)$-cube. This led us to the following theorem:
\begin{theorem}\label{thm:main}
\[
f(n) \leq 4n^{2/3} + o\left(n^{2/3}\right).
\]
\end{theorem}

Recall that the \emph{dimension} $\dim(P)$ of a poset $P$
is the least integer $d$ such that elements of $P$ can be embedded into $\setR^d$ in such a way that $x< y$ in $P$ if and only if
the point of $x$ is below the point of $y$ with respect to the product order on $\setR^d$.
Equivalently, the dimension of $P$ is the least $d$ such that there are $d$ linear extensions of $P$ whose intersection is $P$.
By convention, whenever we say a poset is $d$-dimensional, we mean its dimension
is at most $d$.

Reiniger and Yeager~\cite{Reiniger2016} also studied the guaranteed size of the
largest $d$-dimensional subposet of poset on $n$ elements.
Let $f_d(n)$ be the largest integer such that every poset on $n$ elements has a $d$-dimensional subposet on $f_d(n)$ elements.
They proved, in particular, that $f_d(n) = \O(n^{g})$, where $g = \log_{2d+2} (2d+1)$.

Let $[n]$ denote $\set{0, 1, \ldots, n - 1}$.
By the $\mathbf{n^d}$-grid we mean the poset on the ground set $[n]^d$ with the natural
product order, i.e.~$(x_1, x_2, \ldots, x_d) \leq (y_1, y_2, \ldots, y_d)$ if $x_i \leq y_i$ for all $i$.
Note that the $\mathbf{n^d}$-grid is a $d$-dimensional poset. Moreover, it is easy to see
that the $\mathbf{n^{d+1}}$-grid contains as a subposet the $\mathbf{n^d}$-grid -- simply fix one
coordinate to an arbitrary value.
We prove that this is asymptotically the largest
$d$-dimensional subposet of the $\mathbf{n^{d+1}}$-grid.
For $d \leq 7$, this observation improves on the best known upper bound for the
asymptotics of $f_d(n)$.

\begin{theorem}\label{thm:general-d}
\[
f_d(n) = \O\left(n^{\frac{d}{d+1}}\right).
\]
\end{theorem}

In order to show this we apply a multidimensional version of the theorem
by Marcus and Tardos~\cite{Marcus2004} saying that the number of $1$-entries in
an $n\times n$ $(0,1)$-matrix that avoids a fixed permutation matrix $P$ is
$\O(n)$.
The multidimensional version was proved by Klazar and Marcus~\cite{Klazar2007},
and then independently by Methuku and Pálvölgyi~\cite{Methuku2017}, who applied
it to another extremal problem related to subposets, i.e.~they proved that for
every poset~$P$ the size of any family of subsets of $[n]$ that does not contain
$P$ as a subposet is at most $\O\left(\binom{n}{\floor{n/2}}\right)$.

\section{Dimension two}

If we ignore a multiplicative constant, Theorem~\ref{thm:main} becomes a special
case of Theorem~\ref{thm:general-d}. Still, we provide a short and simple proof
of Theorem~\ref{thm:main}, as we believe it might provide a better insight to
the core of the problem.

\begin{proof}[Proof of Theorem~\ref{thm:main}]
First we argue for values of $n$ such that $n=r^3$ for some $r\in\mathbb{N}$.
Then at the end of the proof we address the general case.

  Let $\mathbf{C_r}$ be the poset with the ground set $[r]^3$, where
  $(x_1, y_1, z_1) \leq_{\mathbf{C_r}} (x_2, y_2, z_2)$ if
  \[
    (z_1 \leq z_2)\quad \text{and}\quad
    (y_1 < y_2 \text{ or } (y_1=y_2 \text{ and } x_1=x_2)),
  \]
see~Figure~\ref{fig:C0}.

  \begin{figure}[h]
	\centering
    \ifspringer
      \includegraphics[scale=0.7]{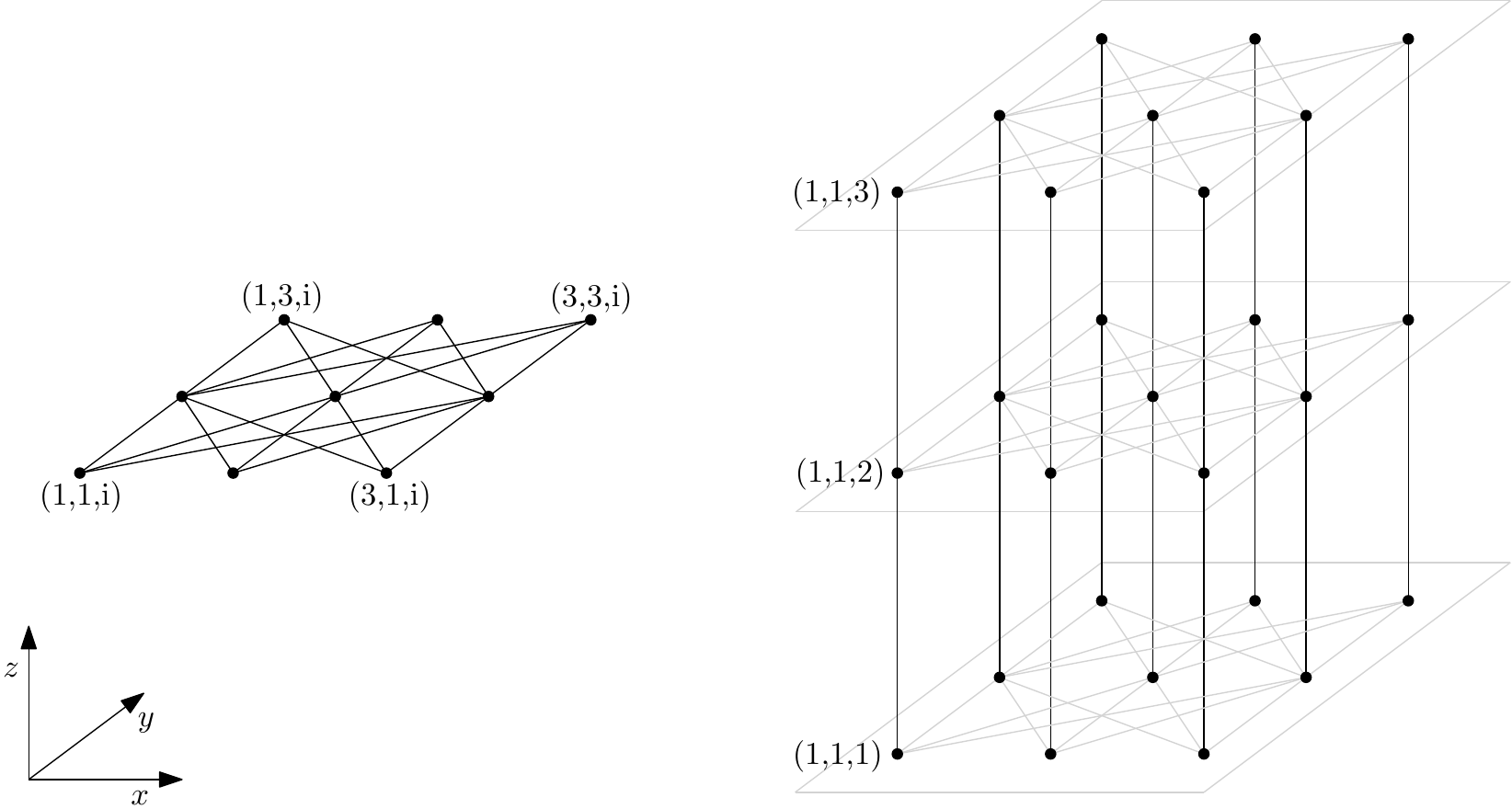}
    \else
      \includegraphics[scale=0.8]{warstwa}
    \fi
    \caption{\label{fig:C0}A subposet of $\mathbf{C_3}$ composed of all elements with $z$-coordinate equal to~$i$ (on the left) and the poset $\mathbf{C_3}$ itself (on the right).}
  \end{figure}

  Consider any subposet $S$ of $\mathbf{C_r}$ such that $|S| \geq 4r^2$.
  We will prove that $\pdim{S} > 2$ by showing that $S$ contains as a
  subposet the poset\footnote{This is one of the $3$-irreducible
  posets, which are listed in~\cite{Trotter1992}.} of dimension $3$
  presented in Figure~\ref{fig:spider}.

  \begin{figure}[h]
    \begin{center}
    \includegraphics[scale=0.8]{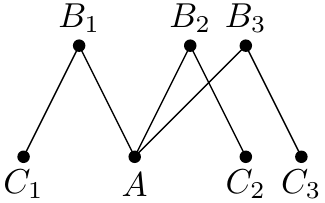}
    \end{center}
    \caption{\label{fig:spider} A poset of dimension $3$ found in any subposet of $\mathbf{C_r}$ of size at least $4r^2$.}
  \end{figure}

Let $S^1$ be the poset obtained from $S$ by removing every element $(x,y,z)$ such that $S$ contains
no element $(x,y,z')$ with $z' < z$.
Note that $|S^1| \geq 3r^2$, as for every pair $(x,y) \in [r]^2$ at most one element is removed.
Now by the pigeonhole principle, we get that
$S^1$ contains a subposet $S^2$ on at least $3r$ elements
such that all elements of $S^2$ have the same $z$-coordinate.

Let $A$ be any point in $S^2$ with the minimal $y$-coordinate
and let $S^3$ be the subposet of $S^2$ obtained by removing all points with the same
$y$-coordinate as $A$.
As there can be at most $r$ points with the same $y$-coordinate, $|S^3| \geq 2r$.
By the pigeonhole principle for $r-1$ containers,
$S^3$ contains three points with the same $y$-coordinate, say $B_1 = (x_1, y, z)$, $B_2 = (x_2, y, z)$, $B_3 = (x_3, y, z)$.
Thanks to the removal rule that led to the creation of $S^1$,
the poset $S$ contains points $C_1 = (x_1, y, z_1)$, $C_2 = (x_2, y, z_2)$, $C_3 = (x_3, y, z_3)$ for some $z_1,z_2,z_3 < z$.

  One can easily verify that the subposet $\set{A, B_1, B_2, B_3, C_1, C_2, C_3}$
  of $S$ is the poset in Figure~\ref{fig:spider}.
  Since it has dimension $3$, we have $\pdim{S} > 2$, which concludes the proof for $n$ being a perfect cube.

  Now, fix any $n \in \N$, and let $r = \ceil{\sqrt[3]{n}}$. Note that $f$ is
  non-decreasing, thus
  \[ f(n) \leq f(r^3) \leq 4r^2 \leq 4(\sqrt[3]{n} + 1)^2 =
     4n^{2/3} + o\left(n^{2/3}\right).\]
\end{proof}

With a more tedious analysis, which involves one more forbidden subposet
and removal of both lowest and highest $z$-coordinate
points in each $(x, y)$-column, we can prove a slightly stronger upper bound,
i.e.~$f(n) \leq 3n^{2/3} + o\left(n^{2/3}\right)$. However, we do not know how to
improve on the asymptotics of $f$.

\section{Higher dimensions}

In this section we prove Theorem~\ref{thm:general-d}.
In order to do this we apply a multidimensional version of the theorem by Marcus
and Tardos~\cite{Marcus2004}, proved by Klazar and Marcus~\cite{Klazar2007}.
First, we recall their result. The original terminology can be simplified because
our argument does not use arbitrary sized matrices and we can focus only on
multidimensional analogues of square matrices.

We call a subset of $[n]^d$ a \emph{$d$-dimensional $(0,1)$-matrix}.

For two $d$-dimensional $(0,1)$-matrices $A \subset [n]^d$ and
$B \subset [k]^d$, we say that $A$ \emph{contains} $B$ if there exist $d$
increasing injections $h_i : [k] \to [n]$, $i \in \{1, 2, \ldots, d\}$,
such that
\[
\text{if }(x_1, x_2, \ldots, x_d) \in B, \text{ then }
  (h_1(x_1), h_2(x_2), \ldots, h_d(x_d)) \in A,
\]
for all $(x_1, x_2, \ldots, x_d) \in [k]^d$.
Otherwise, we say that $A$ \emph{avoids} $B$.

We say that $A \subset [n]^d$ is a \emph{$d$-dimensional permutation of $[n]$}
\[
|A|=n \quad \text{and} \quad \forall_{\substack{x, y \in A\\x \neq y}}
                               \forall_{i\in\{1, 2, \ldots, d\}}\
                               x_i \neq y_i.
\]
In other words, the size of the projection of $A$ onto the $i$-th dimension
equals $n$ for each $i \in \{1, 2, \ldots, d\}$.

\begin{theorem}[Klazar--Marcus~\cite{Klazar2007}]
\label{thm:avoidance}
For every fixed $d$-dimensional permutation $P$ the maximum number of elements
of a $d$-dimensional matrix $A \subset [n]^d$ that avoids $P$ is $\O(n^{d-1})$.
\end{theorem}

Now we are ready to prove the following statement, which clearly implies
Theorem~\ref{thm:general-d}.

\begin{theorem}
The largest $d$-dimensional subposet of the $\mathbf{n^{d+1}}$-grid has
$\O(n^d)$ elements.
\end{theorem}

\begin{proof}
We fix any poset of dimension $d+1$, e.g.~the standard example $S_{d+1}$,
i.e.~the inclusion order of singletons and $d$-element subsets of $[d+1]$.
Now, we fix a realizer of $S_{d+1}$
of size $d+1$, i.e.~a set of $d+1$ linear orders $\{L_1, L_2,\ldots,L_{d+1}\}$
such that $L_1 \cap L_2 \cap \cdots \cap L_{d+1} = S_{d+1}$.
Finally, we construct a $(d+1)$-dimensional permutation
$P \subset [2(d+1)]^{d+1}$ such that $(x_1, x_2, \ldots, x_{d+1}) \in P$ if and
only if there exists $x \in S_{d+1}$ such that $x$ is the $x_i$-th element of
$L_i$ for each $i \in \{1, 2, \ldots, d+1\}$. Note that the natural product
order of elements of $P$ is isomorphic to $S_{d+1}$.

Now, take any $d$-dimensional subposet of the $\mathbf{n^{d+1}}$-grid and denote by $A$ the
set of its elements. In particular, the subposet does not contain $S_{d+1}$ as
a subposet. Note that it implies that $A$ avoids $P$,
thus by Theorem~\ref{thm:avoidance} the size of the subposet is $\O(n^d)$.
\end{proof}

Note that the proof above does not exploit any specific properties of the
standard example, apart from its dimension. In particular, it implies that
every $(d+1)$-dimensional poset $P$ can be found in every subposet of the
$\mathbf{n^{d+1}}$-grid of size $\Omega(n^d)$, with the constant hidden in the
asymptotic notation depending on the choice of $P$.

\section*{Acknowledgments}

We send thanks to Wojciech Samotij and Dömötör Pálvölgyi for pointing us to
useful references.

\bibliographystyle{plain}
\bibliography{paper}

\end{document}